\documentclass[11pt]{article}
\usepackage{amsmath,amsthm,amsfonts,amssymb,amscd, amsxtra}
\usepackage{graphicx,multirow}
\usepackage{float}
\usepackage[margin=2.5cm,nohead]{geometry}
\usepackage{amsfonts}
\usepackage{amsmath}
\usepackage{amsbsy}
\usepackage{amsxtra}
\usepackage{latexsym}
\usepackage{amssymb}
\usepackage{url}
\usepackage{enumerate}
\usepackage{amscd}
\usepackage[active]{srcltx}
\usepackage{cite}

\newtheorem{theorem}{Theorem}
\newtheorem{lemma}{Lemma}

\newtheorem{proposition}{Proposition}
\newtheorem{remark}{Remark}
\newtheorem{definition}{Definition}
\newtheorem{example}{Example}

\DeclareMathOperator{\diag}{diag}
\DeclareMathOperator{\sgn}{sgn}

\newcommand{\R}{\mathbb R}

\begin{document}

\title{On the global convergence of the inexact semi-smooth Newton method for absolute value equation\thanks{IME/UFG, Avenida Esperança, s/n Campus Samambaia,  Goi\^ania, GO, 74690-900, Brazil. This work was partially supported by CAPES-MES-CUBA 226/2012 and UNIVERSAL FAPEG/CNPq projects.}}

\author{
J.Y. Bello Cruz\thanks{This author was partially supported by CNPq Grants 303492/2013-9, 474160/2013-0 (e-mail:{\tt
yunier@ufg.br}).}
\and
O. P. Ferreira\thanks{This author was  partially supported by FAPEG, CNPq Grants 4471815/2012-8,  305158/2014-7  (e-mail:{\tt
orizon@ufg.br}).}
\and
L. F. Prudente\thanks{This author was  partially supported by  FAPEG  (e-mail:{\tt lfprudente@ufg.br}).  }  
 }

\maketitle

\begin{abstract}
\noindent In this paper,   we  investigate global convergence properties of the inexact nonsmooth Newton method for solving the system of absolute value equations (AVE). Global $Q$-linear convergence is established under suitable assumptions. Moreover, we present  some numerical experiments designed to investigate the practical viability of the proposed scheme.
 
\medskip
 
\noindent
{\bf Keywords:}  Absolute value equation,   inexact semi-smooth Newton method, global convergence,  numerical experiments.

\medskip

\noindent {\bf Mathematical Subject Classification (2010):} Primary 90C33;
Secondary 15A48.

\end{abstract}

\section{Introduction}
Recently, the problem of finding a solution of the system of absolute value equations (AVE)
\begin{equation} \label{eq:vae}
Ax-|x|=b, 
\end{equation}
where  \(A\in \R^{n \times n}\) and    $b\in \R^n \equiv \R^{n \times 1}$, has been   received much attention from optimization community.  It is currently  an active  research topic, due to its broad application to many subjects.  For instance, linear complementarity problem, linear programming or convex quadratic programming can be equivalently reformulated in  the form  of \eqref{eq:vae} and thus solved as absolute value equations; see \cite{Mangasarian2014, Prokopyev2009, Mangasarian2007, Rohn2004}.
As far as we know,  since   Mangasarian and Meyer \cite{MangasarianMeyer2006}  established  existence results for this class of absolute value equations \eqref{eq:vae},  the interest for this subject has increased substantially;  see \cite{ Mangasarian2015,  Mangasarian2014-2, Rohn2014} and reference therein.     

Several   algorithms have been designed to solve the systems of AVEs involving  smooth, semi-smooth and Picard techniques; see  \cite{Caccetta2011, Iqbal2015, Salkuyeh2014,  Mangasarian2013, Zhang2009}.  In   \cite{Mangasarian2009},   Mangasarian applied the nonsmooth Newton method for solving AVE  obtaining   global $Q$-linear   convergence  and showing  its  numerical effectiveness.  However, each semi-smooth Newton iteration  requires the exact solution of a linear system,  which  has an undesired effect on the computational performance of   this method.   The exact solution of the linear system, at each iteration of the method,  can be  computational expensive and may not be justified.  A well known alternative is to solve the linear systems involved   approximately.  A bound for the relative error tolerance to  solve subproblem guaranteeing global $Q$-linear convergence  will arise very clearly  in the present work.
Besides,  the inexact analysis support the efficient computational  implementations of the exact schemes, for this reason they are important and necessaries. Following  the ideas  of \cite{Dembo1982} and  \cite{Mangasarian2009},  we  use the  inexact nonsmooth Newton method for solving absolute value equations and present  some computational experiments designed to investigate its practical viability. 


The paper is organized as follows. The next subsection presents some notations and preliminaries that will
be used throughout the paper. Section \ref{sec:ssnm2} devotes the definition of the inexact Newton method and its global $Q$-linear convergence. Section \ref{sec:ctest}
provides an exhaustive discussion of the computational results of the inexact Newton method when it is compared with the exact one. We complete the paper with some conclusion for further study.
\subsection{Preliminaries} \label{sec:int.1}

In this section we present the notations and some auxiliary results used throughout the paper.
Let $\R^n$ be  denote the $n$-dimensional Euclidean space and  $\|\cdot\|$ a  norm.    The  $i$-th component of a vector $x
\in\R^n$ is denoted by $x_i$ for every $i=1,\ldots,n$. For $x\in \R^n$, $\sgn(x)$ denotes a vector with components equal to $-1$, $0$ or $1$ depending on whether the corresponding component of the
vector $x$ is negative, zero or positive.   Denote  $|x|$  the vectors with  $i$-th component equal to   $|x^i|$.
The set of all $n \times n$ matrices with real entries is denoted by $\R^{n \times n}$ and    $\R^n\equiv \R^{n \times 1}$.  The matrix ${\rm Id}$ denotes the $n\times n$ identity matrix.  If $x\in \R^n$ then $\diag (x)$ will denote an  $n\times n$  diagonal matrix with $(i,i)$-th entry equal
to $x^i$, $i=1,\dots,n$. For an $M \in \R^{n\times n}$ consider the norm defined by  $\|M\|:=\max  \{\|Mx\|~:~  x\in \R^{n}, ~\|x\|=1\}$. This definition implies
\begin{equation} \label{eq:np}
\|Mx\|\leq \|M\|\|x\|, \qquad \|L+M\|\le\|L\|+\|M\|, \qquad \|LM\|\le\|L\|\|M\|, 
\end{equation}
for any  matrices \(L, M \in \R^{n\times n}\) and \(x  \in \R^n\); see, for instance,   Chapter 5 of  \cite{HornJohnson1990}.  The next useful result was proved in  2.1.1,  page 32  of \cite{Ortega1990}.
\begin{lemma}[Banach's Lemma] \label{lem:ban}
Let $E \in \R^{n\times n}$. If  $\|E\|<1$,  then the matrix ${\rm Id}-E$
is invertible and  $ \|[{\rm Id}-E]^{-1}\|\leq 1/\left(1-
\|E\|\right). $
\end{lemma}
We end this section  quoting  the following result from combination of Lemma~2 of   \cite{Mangasarian2009} and Proposition 3 of \cite{MangasarianMeyer2006}, which gives us a sufficient condition for the invertibility of matrix  \(A-D(x)\)  for all \(x\in \mathbb{R}^n\) and existence of solution for AVE in \eqref{eq:vae}.
\begin{lemma}\label{le:nonsingGC}
Assume that the singular values of the matrix \(A\in \R^{n \times n}\) exceed \(1\). Then the matrix \(A-D(x)\)   is invertible for all \(x\in \mathbb{R}^n\). Moreover,   AVE in \eqref{eq:vae} has unique solution  for any $b\in \R^n$.
\end{lemma}
\noindent The following proposition  was proved in Proposition 4 of \cite{MangasarianMeyer2006}.
\begin{proposition} \label{pr:mm}
Assume that \(A \in \R^{n\times n}\)  is   invertible.    If \(\left\|A^{-1}\right\|<1 \)  then  AVE in \eqref{eq:vae}   has unique solution  for any $b\in \R^n$.
\end{proposition}

\section{Inexact Semi-smooth Newton Method} \label{sec:ssnm2}
The {\it exact semismooth Newton method} \cite{LiSun93}  for  finding the zero of the semismooth function 
\begin{equation} \label{eq:fuc0} 
F(x):=Ax-|x|-b, 
\end{equation}
with starting point   \(x_0\in \mathbb{R}^n\),  is  formally  defined by 
\[
F(x_k)+ V_k\left(x_{k+1}-x_{k}\right)=0,  \qquad   V_k \in \partial F(x_k), \qquad k=0,1,\ldots,
\]
where    \( \partial F(x)\) denotes   the  Clarke generalized subdiferential  of \(F\) at \(x\in \mathbb{R}^n\); see~\cite{Clarke1990}. Letting 
\begin{equation} \label{eq:subgrad0}
D(x):=\mbox{diag}(\mbox{sgn}(x)), \qquad x\in  \R^n, 
\end{equation}
we obtain from \eqref{eq:fuc0} that    \(A-D(x) \in  \partial F(x)\). Hence,  the    {\it exact  semi-smooth Newton  method}    for solving the   AVE in  \eqref{eq:vae},   which was proposed by Mangasarian~\cite{Mangasarian2009}, generates a sequence   formally   stated as
\begin{equation} \label{eq:nmqc0}
\left[A-D(x_k)\right]x_{k+1}=b.
\end{equation}
 To solve  \eqref{eq:vae}, following  the ideas of \cite{Dembo1982}, we propose  an   {\it inexact semi-smooth Newton method},  starting at   \(x_0\in \mathbb{R}^n\) and {\it residual relative error tolerance} \(\theta \geq 0\),  by 
\begin{equation} \label{eq:newtonc2}
	\left\| \left[ A-D(x_k)\right]x_{k+1}-b\right\| \leq  \theta \left \|F(x_k) \right\|,  \qquad k=0,1,\ldots.
\end{equation}
Note that, in  absence of errors, i.e., \(\theta=0\),  the above iteration  retrieves   \eqref{eq:nmqc0}. In the next section we  analyze the convergence  properties of \(\{x_k\}\) generated by  the   {\it inexact semi-smooth Newton  method}. 
\subsection{Convergence Analysis} \label{sec:ssnm2*}
To establish convergence of the sequence \(\{x_k\}\),  generated by  \eqref{eq:newtonc2},  we need some auxiliary results and basic definitions.

The outcome of an inexact Newton iteration is any point satisfying some error tolerance. Hence, instead of a mapping for the inexact Newton iteration, we shall deal with a \emph{family} of mappings describing all possible inexact iterates.
\begin{definition} \label{def:inire}
  For \( \theta\geq 0\), \(\mathcal{N}_\theta\) is the family of maps  \(N_\theta:\mathbb{R}^n\to  \mathbb{R}^n\) such that
  \begin{equation}  \label{eq:in.map}
  \left\|\left[ A-D(x)\right]N_\theta(x)-b\right\| \leq  \theta \left \|F(x) \right\|, \qquad \forall ~x \in \mathbb{R}^n.
  \end{equation}
\end{definition}
If \(A-D(x)\)   is invertible for all \(x\in \mathbb{R}^n\), then the family $\mathcal{N}_0$ has a single element, namely,  the exact Newton iteration map \(N_{0}:\mathbb{R}^n \to \mathbb{R}^n\) defined by 
\begin{equation} \label{NF}
  N_{0}(x)=\left[A-D(x)\right]^{-1} b.
\end{equation}
Trivially, if $0\leq \theta \leq \theta '$ then
$\mathcal{N}_0\subseteq\mathcal{N}_\theta\subseteq\mathcal{N}_{\theta '}$. Hence $\mathcal{N}_\theta$ is non-empty for all $\theta\geq 0$.
\begin{remark} \label{eq:fp}
Let \(x\in \mathbb{R}^n\) and  \(A-D(x)\)   be  invertible.  For any $0<\theta<1 $ and $N_\theta\in\mathcal{N}_\theta$, \(N_\theta(x)=x\) if and only if  \(F(x)=0\). This means that the fixed points of the inexact Newton iteration map $N_\theta$ are the same fixed points of the \emph{exact} Newton
  iteration map.
\end{remark}
\begin{lemma} \label{le:cl}
Assume that  \(A-D(x)\)   is invertible for all \(x\in \mathbb{R}^n\).  Let  \( \theta \geq 0\) and \(N_\theta \in \mathcal{N}_\theta\). If  \(F(x_*)=0\)  then for each \( x\in \mathbb{R}^n\) there holds
\[
\left\|N_\theta(x)- x_*\right\|\leq  \left\| [A-D(x)]^{-1}\right\|\big[\theta  \left( \|A-D(x)\| +2\right) +  2\big] \|x-x_*\|.
\]
\end{lemma}
\begin{proof}
Let  \( x\in \mathbb{R}^n\).  After  simple  algebraic manipulations and taking into account that   \(F(x_*)=0\) and  \(\left[A-D(x)\right]x=Ax-|x|\),  we obtain
\[
N_\theta(x)- x_*=  \left[A-D(x)\right]^{-1} \big(\left[A-D(x)\right] N_\theta(x)-b +  \left[ F(x_*)-F(x)- \left[A-D(x)\right](x_*-x) \right]\big).
\]
Taking norm in both side of  above equality and using its properties in    \eqref{eq:np},  we conclude that 
\[
\left\|N_\theta(x)- x_*\right\|\leq \left\| [A-D(x)]^{-1}\right\| \left(\left\|\left[A-D(x)\right] N_\theta(x)-b\right\| +    \left\| F(x_*)-F(x)- \left[A-D(x)\right](x_*-x) \right\|\right).
\]
The combination of  Definition~\ref{def:inire} and   last inequality implies
\begin{equation} \label{eq:ini}
\left\|N_\theta(x)- x_*\right\|\leq \left\|[A-D(x)]^{-1}\right\|\big(\theta \left \|F(x) \right\| +   \left\| F(x_*)-F(x)- \left[A-D(x)\right](x_*-x) \right\|\big).
\end{equation}
On the other hand,  since   $F(x_*)=0$,  direct  algebraic manipulations  give us
\[ 
 F(x)=\left[A-D(x)\right](x- x_*) -  \left[ F(x_*)-F(x)- \left[A-D(x)\right](x_*-x) \right].
\]
Thus, taking norm in  both sides of  last equality and by   triangular inequality,   we have
\begin{equation}\label{eq-intermediaria}
\left\| F(x)\right\| \leq \left\|A-D(x)\right\| \|x- x_*\| +  \left\| F(x_*)-F(x)- \left[A-D(x)\right](x_*-x) \right\|.
\end{equation}
Now,  using that   \(D(x)x=|x|\), \(D(x_*)x_*=|x_*|\) and  definition in \eqref{eq:fuc0},   after some algebras,   we can write 
\[
F(x_*)-F(x)- \left[A-D(x)\right](x_*-x)= -(|x_*|-|x|)+D(x)(x_*-x).
\]
Since \(\|D(x)\|\leq 1\),   taking norms on both sides of last equality, and using the  triangular inequality together with properties of the norm, namely, the first one stated in \eqref{eq:np} yields
\begin{equation}\label{eq:partelinear}
\left\| F(x_*)-F(x)- \left[A-D(x)\right](x_*-x)\right\| \leq 2\|x_*-x\|.
\end{equation}
Therefore,  combining the last inequality and  \eqref{eq-intermediaria},    we  conclude that 
$$
\left\| F(x)\right\| \leq \left\|A-D(x)\right\| \|x- x_*\| +  2\|x_*-x\|.
$$
Substituting the last inequality and \eqref{eq:partelinear} into  \eqref{eq:ini}, we obtain 
\begin{equation} \label{eq:lntheta}
\left\|N_\theta(x)- x_*\right\|\leq \left\|[A-D(x)]^{-1}\right\|\big[\theta \big(\left\|A-D(x)\right\| \|x- x_*\|  +2\|x-x_*\|\big) +   2\|x-x_*\|\big], 
\end{equation}
which is equivalent to the desire inequality.
\end{proof}
Let  $0\leq \theta <1$ and  $\mathcal{N}_\theta$   in Definition~\ref{def:inire}.  Consider the sequence  \(\{x_k\}\) defined in  \eqref{eq:newtonc2}. Thus, there exists $N_{\theta} \in \mathcal{N}_\theta$  such that 
\begin{equation} \label{eq:NFSs}
x_{k+1}=N_{\theta}(x_k),\qquad k=0,1,\ldots \,.
\end{equation}
Now, we are ready to prove  the  two main results of this section.    
\begin{theorem}\label{th:mrq}
Let  \(A\in \R^{n \times n}\), \(b\in \R^n\) and $0\leq \theta <1$.  Assume that  \(A-D(x)\)   is invertible for all \(x\in \mathbb{R}^n\). Then,   the inexact semi-smooth Newton sequence \(\{x_k\}\),   given by \eqref{eq:newtonc2},  with any  starting  point   \(x_0\in \mathbb{R}^n\) and residual relative error tolerance $\theta$, 
 is well defined.
 Moreover, if 
\begin{equation} \label{eq:wdnm}
\left\|\left[A-D(x)\right]^{-1}\right\|< \frac{1}{2},   \qquad  \forall ~x \in \mathbb{R}^n, 
\end{equation}
then,  AVE in  \eqref{eq:vae} has unique solution. Additionally, if 
\begin{equation} \label{eq:theta0}
 0\leq \theta < \frac{1-2\left\| [A-D(x)]^{-1}\right\|}{\left\| [A-D(x)]^{-1}\right\|\left( \|A-D(x)\|+2\right)},   \qquad  \forall ~x \in \mathbb{R}^n.
\end{equation}
 then   $ \{x_k\}$   converges $Q$-linearly  to  \(x_*\in \mathbb{R}^n\),  a unique solution  of  \eqref{eq:vae}, as follows 
  \begin{equation} \label{eq:lconv3}
\left\|x_{k+1}- x_*\right\|\leq  \left\| [A-D(x_k)]^{-1}\right\|\left[\theta  \left(\|A-D(x_k)\|+2 \right) +  2\right] \|x_{k}-x_*\|,
\end{equation}
for all \(k=0, 1, \ldots\).  
\end{theorem}
\begin{proof} For any starting point \(x_0 \in \R^{n}\),  by  Definition~\ref{def:inire} and \eqref{eq:newtonc2},   the well-definedness of  \(\{x_k\}\)  follows from invertibility of \(A-D(x)\)  for all \(x\in \mathbb{R}^n\).  The uniqueness of the solution follows from  Proposition~\ref{pr:mm},  by taking $x=0$ in \eqref{eq:wdnm}.  Since \(x_*\) is the solution of \eqref{eq:vae}  we have \( F(x_*)=0\). Hence, using Lemma~\ref{le:cl} and  \eqref{eq:NFSs},  it is immediate to conclude that $\{x_k\}$  also satisfies \eqref{eq:lconv3}.   On the other hand,  using  \eqref{eq:wdnm} and assumption on \(\theta\), i.e,   the inequalities in   \eqref{eq:theta0},   we conclude that
\[
  \left\|[A-D(x_k)]^{-1}\right\|\left[\theta    \left(\|A-D(x_k)\|+2\right) +  2\right]<1, \qquad k=0, 1, \ldots, 
\]
which, taking into account \eqref{eq:lconv3},  implies  that  \(\{x_k\}\) converges  $Q$-linearly   to \(x_*\). 
\end{proof}
\begin{remark}
In the absence of errors, i.e., \(\theta=0\), Theorem~\ref{th:mrq}  retrieves Lemma~6 of \cite{Mangasarian2009} on the exact  semi-smooth Newton method. Emphasizing that our assumption in \eqref{eq:wdnm} is weaker than the used by Mangasarian in Lemma~6 of \cite{Mangasarian2009}, i.e., $\|\left[A-D(x)\right]^{-1}\|< 1/3$ for all $x\in \mathbb{R}^n$. 
\end{remark}
\begin{theorem}\label{th:mrq*}
Let \(A \in \R^{n\times n}\)  be  a invertible matrix, $b\in \R^n\).  Assume that  
\begin{equation}\label{Inv-1/3}
\left\|A^{-1}\right\|<\frac{1}{3}.
\end{equation} 
Then,  the  inexact semi-smooth Newton sequence \(\{x_k\}\),   for solving   \eqref{eq:vae},  with  any starting  point   \(x_0\in \mathbb{R}^n\) and residual relative error tolerance $\theta\geq 0$, 
 is well defined.  Moreover, if 
\begin{equation} \label{eq:theta}
 0\leq \theta < \frac{1-3\left\| A^{-1}\right\|}{ \left\| A^{-1}\right\|(\left\| A\right\|+3)}, 
\end{equation}
 then   $ \{x_k\}$   converges $Q$-linearly  to  \(x_*\in \mathbb{R}^n\),  the unique solution  of  \eqref{eq:vae}, as follows 
\begin{equation} \label{eq:conv}
\left\|x_{k+1}- x_*\right\|\leq  \frac{\left\|A^{-1}\right\|}{1-\left\|A^{-1}\right\|}\left[\theta  \left(\|A\|+3\right)  +  2\right] \|x_k-x_*\|, \qquad k=0, 1, \ldots.
\end{equation}
\end{theorem}
\begin{proof}
Let  \(x \in \mathbb{R}^n\). Since  \(\left\|A^{-1}\right\|<1/3 \),  by  \eqref{eq:np}, we have  \(\|A^{-1}D(x)\| <1/3<1 \).  Thus,     from Lemma~\ref{lem:ban}  it follows that  \( {\rm Id}-A^{-1}D(x)\) is invertible. Since \(A \in \R^{n\times n}\)  is   invertible and 
 \[
  A-D(x)=A \left[{\rm Id}-A^{-1}D(x)\right],
 \]
 we conclude  that  \( A-D(x)\) is invertible. Hence \(\{x_k\}\)  is well defined,  for any starting  point   \(x_0\in \mathbb{R}^n\).   
  
Since  \(\left\|A^{-1}\right\|<1 \),  the uniqueness follows from  Proposition~\ref{pr:mm}.  Let  \(x_*\in \mathbb{R}^n\) be the  unique solution  of  \eqref{eq:vae}.   Since  \(A-D(x)\)   is invertible for all \(x\in \mathbb{R}^n\), we may apply Lemma~\ref{le:cl} to obtain 
\[
\left\|N_\theta(x_k)- x_*\right\|\leq  \left\| [A-D(x_k)]^{-1}\right\|\left[\theta  \left( \|A-D(x_k)\| +2\right) +  2\right] \|x_k-x_*\|,  \qquad k=0, 1, \ldots.
\]
 On the other hand, it is easy to see that  
\[
\left\| [A-D(x_k)]^{-1}\right\| = \left\| [{\rm Id}-A^{-1}D(x_k)]^{-1}A^{-1}\right\| \leq  \left\| [{\rm Id}-A^{-1}D(x_k)]^{-1}\right\| \left\|A^{-1}\right\|,  \qquad k=0, 1, \ldots. 
\]
Thus, combining two last inequality we conclude that 
\[
\left\|N_\theta(x_k)- x_*\right\|\leq  \left\| [{\rm Id}-A^{-1}D(x_k)]^{-1}\right\| \left\|A^{-1}\right\| \left[\theta  \left( \|A-D(x_k)\| +2\right) +  2\right] \|x_k-x_*\|,  \qquad k=0, 1, \ldots.
\]
Hence, since  \(\|A^{-1}D(x_k)\|\leq \|A^{-1}\|<1/3 \), the  inequality \eqref{eq:conv} follows  from  Lemma~\ref{lem:ban} and considering that \(\|A-D(x_k)\|\leq \|A\|+1\).  Finally, we conclude that  $ \{x_k\}$   converges $Q$-linearly  to  \(x_*\), by taking into account \eqref{Inv-1/3} and  \eqref{eq:theta}. 
\end{proof}
\begin{remark}
In the absence of errors, i.e., \(\theta=0\), Theorem~\ref{th:mrq*}  retrieves Proposition~7 of  \cite{Mangasarian2009} on the exact  semi-smooth Newton method. Emphasizing that our assumption in \eqref{Inv-1/3} is weaker than the used by Mangasarian in Proposition~7 of  \cite{Mangasarian2009}, i.e., $\|A^{-1}\|< 1/4$.
\end{remark}

Note that the quantity in the right hand side in  \eqref{eq:theta}  is less than \(1/{\rm cond}(A)\), where \({\rm cond}(A)=\|A\|\|A^{-1}\|\). Hence,  for ill conditioned equations, $\theta$ in \eqref{eq:theta} has to be chosen small and hence the precision for solving \eqref{eq:newtonc2} may be high. It is worth to mention that, all ours  results  hold for any  matrix norm satisfying \eqref{eq:np}. Observe that, the upper bound for $\theta$ in \eqref{eq:theta},   for the norms   $||. ||_{1}$ and $\| . \|_{\infty}$   are easy to compute.

In next result,  we discuss a condition for identify a solution of the AVE  throughout   the sequence.
\begin{proposition}\label{prop1}
Assume that  \(A-D(x)\)   is invertible for all \(x\in \mathbb{R}^n\). If $x_{k+1}=x_k$ for some $k$, then $x_{k}$ is solution of \eqref{eq:vae}. 
\end{proposition}
\begin{proof}
Since $x_{k+1}=x_k$, it follows from \eqref{eq:NFSs} that   \(x_{k}=N_{\theta}(x_k)\). Thus,  from Remark~\ref{eq:fp} we conclude that $ F(x_{k})=0$, which implies that $x_{k}$ is solution of \eqref{eq:vae}. 
\end{proof}
If $\theta=0$ then $D(x_k)=D(x_{k+1})$ is a sufficient condition for $x_{k+1}$ be solution of  \eqref{eq:vae}; see \cite{Mangasarian2009}. However, this does not occur in the inexact semi-smooth Newton method defined by \eqref{eq:newtonc2}, as observed in the next example.
\begin{example}
Setting the data of problem \eqref{eq:fuc0} as $A=4\,{\rm Id}$ and $b=e$, where $e$ is the ones vector in $\mathbb{R}^n$, we obtain  $F(x)=4x-|x|-e$. Note that $x_*=(1/3)\,e$ is the unique  solution of problem \eqref{eq:vae} and $\|[A-D(x)]^{-1}\|<1/2$ for all $x\in\mathbb{R}^n$ and $\|A^{-1}\|=1/4<1/3$ within the below assumptions \eqref{eq:wdnm} in Theorem~\ref{th:mrq} and \eqref{Inv-1/3} in Theorem~\ref{th:mrq*}, respectively. Starting at $x_0=e$, the iteration \eqref{eq:newtonc2} leads us
$$
\|[A-D(x_0)]x_{1}-b\|=\left\|3x_1-e\right\|\le\theta\left\|3x_0-e\right\|=\|2\theta e\|,$$ for all $0<\theta<1$. Thus, we can take  $x_1=[(2\theta+1)/3]e$, implying  $D(x_0)=D(x_1)={\rm Id}$ and $$F(x_1)=[A-D(x_1)]\,x_1-b=3x_1-e=2\theta e\neq0.$$ Therefore, $x_1$ is not solution.
\end{example}
The invertibility  of \(A-D(x)\),  for all \(x\in \mathbb{R}^n\),  is sufficient to the  well-definition of the exact and inexact Newton's method. However,   the next example shows that, for the convergence of these methods,  an additional condition on  $A$ must be assumed, for instance,   \eqref{eq:wdnm} and \eqref{Inv-1/3}.
\begin{example}
Consider the  function $F: \mathbb{R}^2 \to \mathbb{R}^2$ defined by  $F(x)=Ax-|x|-b$, where
$$
A=\begin{bmatrix}
1 & -1 \\
3 & -1
\end{bmatrix},  \qquad 
b=\begin{bmatrix}
-1\\
-3 
\end{bmatrix}.
$$
Note that  $\|A^{-1}\|=1,7071...$ and the matrices \(A-D(x)\) are invertible,  for all \(x\in \mathbb{R}^2\). Moreover,   $F$ has $x_*=-e$ as the unique zero, where $e$ is the ones vector in $\R^2$.  Applying exact Newton's method starting with $x_0=e$,  for finding the zero of $F$, the generated sequence oscillates between the points
$$
x_1=\displaystyle \begin{bmatrix}
-\frac{1}{3}\\
1
\end{bmatrix}, \qquad 
x_2=\displaystyle \begin{bmatrix}
1\\
3 
\end{bmatrix}.
$$
\end{example}
\section{Computational Results} \label{sec:ctest}

In order to verify the effectiveness of our approach, we compared the performance of the exact and inexact semi-smooth Newton methods for solving several AVEs.
In a first group of tests, $A$ is supposed to be a large-scale sparse matrix. 
The influence of the condition number and density of $A$ were also investigated.
In many considered cases the performance of the inexact semi-smooth Newton methods is remarkably better than that of the exact one.
All codes were implemented in Matlab 7.11.0.584 (R2010b) and are free available in \textit{https://orizon.mat.ufg.br/admin/pages/11432-codes}.
The experiments were run on a 3.4 GHz Intel(R) i7 with 4 processors, 8Gb of RAM and Linux operating system.
Following, we enlighten some implementation details.

\medskip

\noindent {\bf (i)} {\it Convergence criteria}: The implemented methods stop at iterate $x_k$ reporting ``Solution found'' if AVE is solved with accuracy $10^{-8}$. 
 This means that the 2-norm of $Ax_k-|x_k|-b$ is less than or equal to $10^{-8}$. 
 In some cases we use a different tolerance that will be opportunely reported. The methods also stop reporting failure if the number of iterations exceeds $50$.
 
 \medskip

\noindent {\bf (ii)} {\it Generating random problems}: To construct matrix $A$ we used the Matlab routine {\it sprand}, which generates a sparse matrix with predefined dimension, density and singular values. 
 Firstly, we defined the dimension $n$ and random 
generated the vector of singular values $rc$ from a uniform distribution on $(0 \ \ 1)$. To ensure fulfillment of the hypothesis of Theorem \ref{th:mrq*}, we rescale $rc$ by 
multiplying it by 3 divided by the minimum singular value multiplied by a random number in the interval $(0 \ \ 1)$. Finally, we evoke {\it sprand}. In this case, $A$ is generated by random plane 
rotations applied to a diagonal matrix with the given singular values $rc$. After that, we chose a random solution $x$ and initial point $x_0$ from a uniform distribution on $(-100 \ \ 100)$ and 
computed $b=A x -|x|$.  Note that, since the singular values of $A$ are known, the calculation of the residual relative error tolerance $\theta$ is simple. In particular, we defined $\theta$ as the 
right hand side of \eqref{eq:theta} multiplied by $0.9999$.

\medskip

\noindent {\bf (iii)} {\it Solving linear equations}:  In each iteration of the exact semi-smooth Newton method, the linear system \eqref{eq:nmqc0} must be exactly solved. 
In this case, we used the {\it mldivide} (same as {\it backslash}) command of Matlab. 
In order to take advantages of the structure of the problem, {\it mldivide} performs a study of the matrix of the linear system.
The first distinction the function makes is between dense and sparse input arrays (sparsity must be informed by the user).
Since the random matrix $A$ has no special structure, {\it mldivide} should compute the bandwidth of $A$ and use a banded solver in the case of $A$ is sparse (band density less than or equal to $0.5$), 
or use LU solver, otherwise. LU factorization can be used for both (informed) sparse and dense matrices.
On the order hand, the inexact semi-smooth Newton iteration requires to approximately solve the linear system \eqref{eq:nmqc0} in the sense of \eqref{eq:newtonc2}. 
Matlab has several iterative methods for linear equations. 
In order to choose the most appropriate one, we performed a preliminary test comparing the performance of all of them. We run the inexact semi-smooth Newton method with each 
iterative methods for $50$ problems with $n = 400$ and density equal to $0.01$. The method with the routine {\it lsqr} was the most efficient in {\it all} considered problems.
Therefore, in all forthcoming tests, we used {\it lsqr} as iterative method to approximately solve linear equations. The routine {\it lsqr} is an algorithm for sparse linear equations and sparse least 
squares based on \cite{lsqr}. We highlight that $x_k$ is used as a starting point by {\it lsqr} to solve the linear system of $k+1$ iteration.
In the numerical tests of Section \ref{sparseAVE} where the density of $A$ is approximately equal to 0.003, the matrices of linear systems were reported to be sparse for both exact and inexact semi-smooth 
Newton methods. In Section \ref{denseAVE}, where we analyzed the influence of higher densities, linear systems were solved considering full matrices for both methods.

\medskip


We presented the numerical comparisons using performance profiles graphics; see \cite{pprofile}.
Defined a performance measurement (for example, CPU time), performance profiles are very useful tools to compare several methods on a large set of test problems.
In a simple manner, they allow us to compare efficiency and robustness. 
The efficiency of a method is related to the percentage of problems for which the method was the fastest one, while robustness is related to the percentage of problems for which the method found a solution. 
In a performance profile, efficiency and robustness can be accessed on the extreme left (at $1$ in domain) and right of the graphic, respectively.
In all group of tests, we use CPU time as performance measurement. In order to obtain more accurately the CPU time, we run all test problems in each method ten times and 
we defined the corresponding CPU time as the median of these measurements.
For each problem, we consider that a method is the most efficient if its runtime does not exceed in $5\%$ the CPU time of the fastest one.

\subsection{Large-Scale Sparse AVEs} \label{sparseAVE}

It is well-known that direct methods to solve linear systems can be impractical if the associated matrix is large and sparse. 
Basically, there are two reasons. First, the factorization of the matrix can cause fill-in, losing the sparsity structure and generating additional computational cost.
Second, a floating-point error made in one step spreads further in all following ones. Therefore, it is rarely true in applications that direct methods give the exact solution of linear systems and, 
in some cases, they are not able to find the solution with a desired accuracy. In the forthcoming numerical results, this phenomenon will appear very clearly.
It is opportune to mention that there are many direct methods for large-scale sparse linear systems that seek to alleviate these drawbacks; see, for example, \cite{davis}.
On the other hand, iterative methods do not exhibit these disadvantages because they do not work with matrix factorization and they are self-correcting, in the sense that, 
if floating-point errors affect $x_k$ during the $k$-th iteration, $x_k$ can be considered as a new starting point, not affecting the convergence of the method. 
Consequently, iterative methods are, in general, more robust than the direct ones. 

Seeking to solve an AVE, the exact and inexact Newton iterates deal with a large-scale sparse linear system when matrix $A$ has these characteristics. 
The previous discussion allows us to conjecture that, in this case, the inexact Newton method is better than the exact one.
In an attempt to confirm this intuition, in the first group of tests, we generated $200$ AVEs with 
$n = 10000$ and density of $A$ approximately equal to $0.003$. This means that, only about $0.3\%$ of the elements of $A$ are non null.
In this set of problems, the average condition number of $A$ is approximately $40$ (lowest and largest values is $1.87$ and $1610$, respectively), while the average $\theta$ is of order of 
$10^{-2}$. Figure \ref{figwellcond} shows a comparison, using performance profiles, between exact and inexact semi-smooth Newton methods in this set of test problems.

\begin{figure}
	\centering
		\includegraphics[scale=0.55]{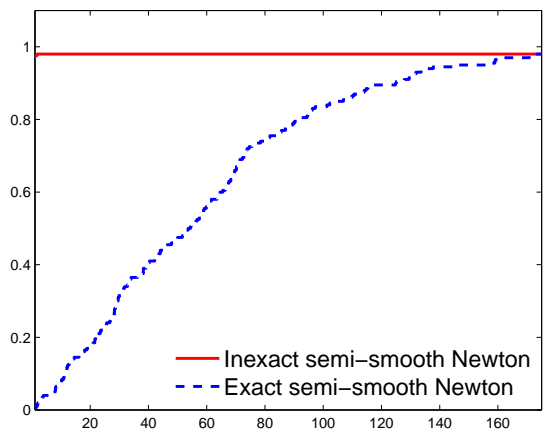}\\

\caption{Performance profile comparing Inexact {\it versus} Exact semi-smooth Newton method for large-scale sparse AVEs with well conditioned matrices.}
\label{figwellcond}
\end{figure}

Analyzing Figure \ref{figwellcond}, we see that, as expected, the inexact method is more efficient than the exact one. However, it may be surprising that the difference is so stark.
Efficiencies of the methods are $97,5\%$ and $5\%$, respectively, for the inexact and exact versions. 
The robustness is $98\%$ for both ones. The methods failed to solve the same four problems. Nevertheless, it is interesting to point out that, for this four problems, both methods 
reached a very close accuracy from the required one to report ``Solution found''.
In this first set of problems, the exact Newton method was faster in just {\it one} problem (21.4s against 35.6s of the inexact method).
Curiously, this is the problem for which the condition number of $A$ is the highest value among all considered ones.
Considering only the others solved problems, the methods spent in total 201s and  4541s for the inexact and exact versions, respectively. 
The inexact method solves a typical problem of this set in less than a second while the exact one takes about 20s (in average, the exact method demands 57 times the runtime of the inexact one).
The average number of iterations was, respectively, 9.6 and 3.4, showing that, as expected, the exact iteration requests greater computational effort than the inexact one.


Since average condition number of the matrices of the first set problems are relatively small (consequently, in average, the matrices are well conditioned),  we investigated the performance of the methods in sets of ill-conditioned matrices. This analysis is pertinent because, as observed, the right hand side of \eqref{eq:theta} is less than the inverse of condition number. So, for ill-conditioned matrices, the iterative solver for linear equations needs to be much required. 
It is not reasonable to set a value for the condition number to claim a matrix as ``ill-conditioned''.
Therefore, in this second phase, we generated two sets of $200$ AVEs. In the first one, the condition numbers 
of all matrices $A$ are greater than or equal to $10^2$ (the average is $4.73 \times 10^2$) while, in the second one, are greater than or equal $10^5$ (the average is $1.72 \times 10^5$). 
We emphasize that in these second phase, we keep $n = 10000$ and the density of $A$ approximately equal to $0.3\%$.
Figure \ref{figill} shows a comparison between exact and inexact semi-smooth Newton methods in this two sets of test problems.

\begin{figure}

\begin{minipage}[b]{0.50\linewidth}

\begin{figure}[H]
	\centering
		\includegraphics[scale=0.55]{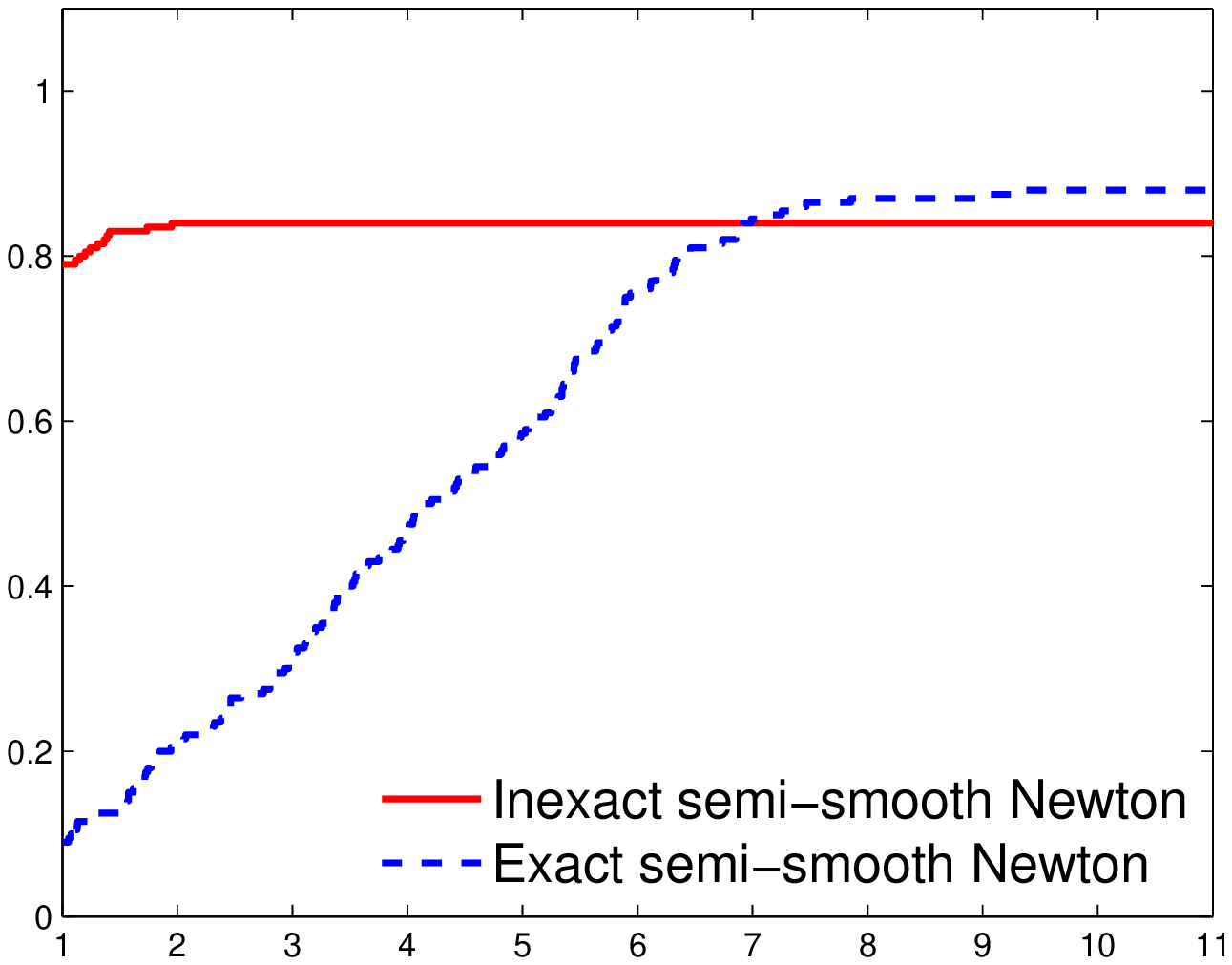}\\
		(a)
\end{figure}

\end{minipage} \hfill
\begin{minipage}[b]{0.50\linewidth}

\begin{figure}[H]
	\centering
		\includegraphics[scale=0.55]{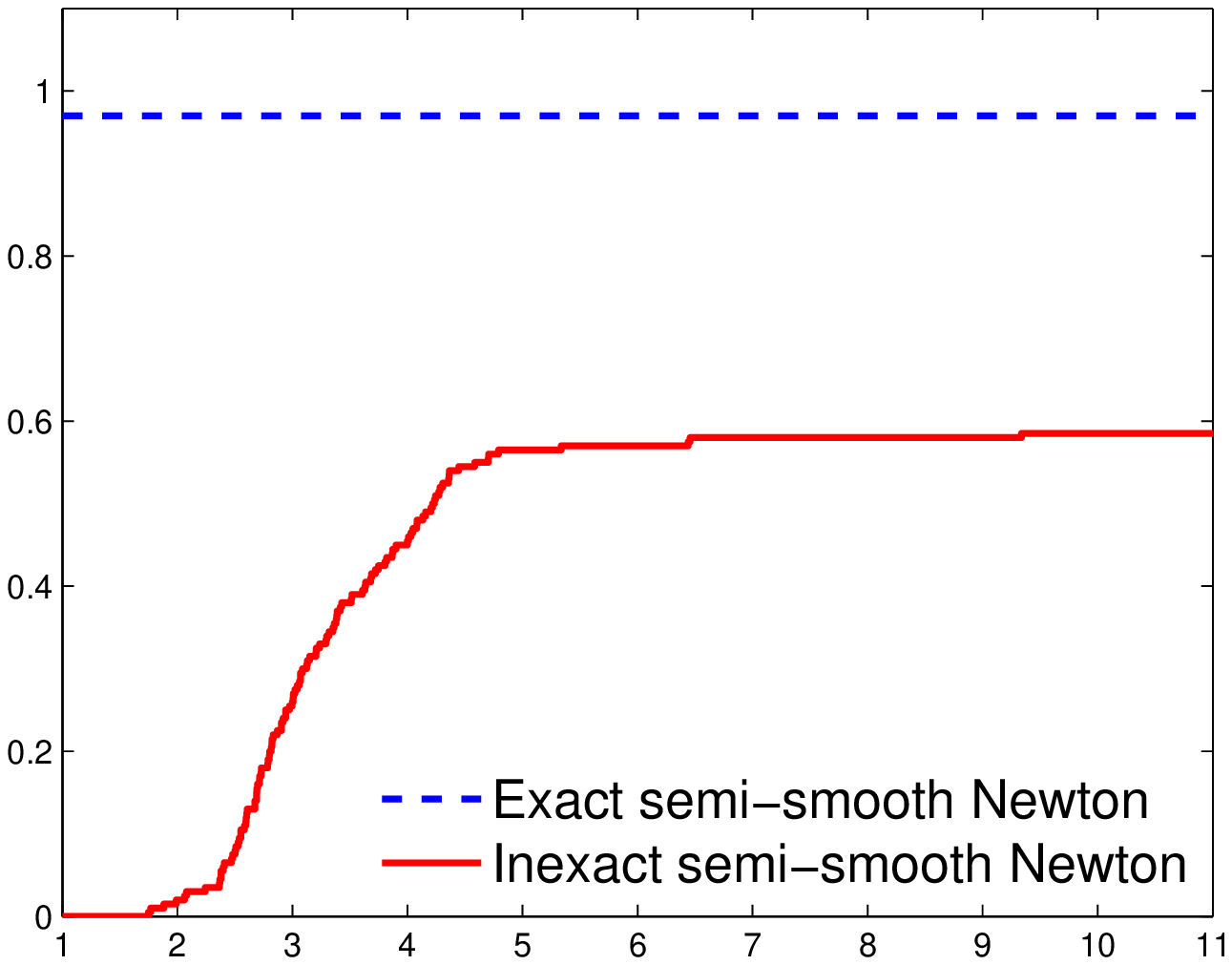}\\
		(b)
\end{figure}

\end{minipage}\hfill

\caption{Performance profile comparing Inexact {\it versus} Exact semi-smooth Newton method for large-scale sparse AVEs with ``ill-conditioned'' matrices: (a) ${\rm cond}(A) \approx 10^2$ and 
(b) ${\rm cond}(A) \approx 10^5$.}
\label{figill}
\end{figure}

In the first set of the second phase, analyzing Figure~\ref{figill}(a), we see that the inexact Newton method is more efficient than the exact one (efficiencies are 79\% and 9\%, respectively), 
while robustness rates are 84\% and 88\%, respectively.
The robustness difference corresponds to 8 problems that were solved by exact method and failed to be solved by the inexact one. Reciprocally, there is no problem. In others 24 problems, both methods 
failed. However, in all cases of failure, regardless of the method, the achieved accuracy was close to the desired one (typically, order of $10^{-8}$). 
This allows us to conclude that both methods were equivalently robust.
In terms of processing time, comparing with the set of problems of the 
first phase, the exact Newton method showed a considerable improvement. Note by  Figure~\ref{figill}(a) that the performance functions intersect each other approximately at 7 in the domain, while in 
Figure~\ref{figwellcond} the intersection takes place approximately at 160. Consequently, the difference between the two methods is much smaller in the case that the condition 
number of $A$ is order of $10^2$ than in the case where it is order of $10^1$.

Consider now the second set of problems where the condition numbers of $A$ are of order of $10^5$. 
Both methods were not able to solve with accuracy $10^{-8}$ an AVE that belongs to this set of ill-conditioned problems.
So, in this specific case, we used the accuracy equal to $10^{-5}$.
For the inexact method, the average value of $\theta$ is $2.85 \times 10^{-6}$.
Figure~\ref{figill}(b) shows the performance of the methods.
As we can note, the exact Newton method achieved much higher performance compared with the inexact one. The exact method was the faster one in {\it all} problems that were solved by both methods. 
The robustness is $59\%$ and $97\%$ for the inexact and exact method, respectively. 
The high efficiency of the exact Newton method is connected with the ability of this method to solve a problem with a moderate accuracy ($10^{-5}$) in a few iterations. The exact method did not need more than 
three iterations in order to solve a problem of this set: 166 problems were solved with 2 iterations, while 28 with 3 iterations.
On the other hand, since $\theta$ is oder of $10^{-6}$, the iterative linear equations solver is much required in an early iteration of the inexact method, generating additional runtime.
The poor robustness of the inexact method can be related with the lack of robustness of the {\it lsqr} routine to deal with ill-conditioned matrices.
It is appropriate to mention that no preconditioner was used in {\it lsqr}. 

Table \ref{tabsummary1} summarizes the results of the numerical experiments of the current section. The column ``${\rm cond}(A)$'' informs the (average) order of the condition number of the matrices $A$, while 
column ``$tol$'' reports the considered accuracy to solve an AVE. 

\begin{table}[H]
{\footnotesize
\begin{center}
\begin{tabular}{cc|cc|cc|} \cline{3-6}
 &  & \multicolumn{2}{|c|}{Inexact Newton method} & \multicolumn{2}{|c|}{Exact Newton method}\\ \hline

\multicolumn{1}{ |c  }{${\rm cond}(A)$} & $tol$ & Efficiency ($\%$) & Robustness ($\%$) & Efficiency ($\%$) & Robustness ($\%$)\\  \hline
\multicolumn{1}{ |c  }{$10^1$} & $10^{-8}$ & 97.5 & 98.0  &  5.0  & 98.0  \\
\multicolumn{1}{ |c  }{$10^2$} & $10^{-8}$ & 79.0   & 84.0  & 9.0  & 88.0  \\
\multicolumn{1}{ |c  }{$10^5$} & $10^{-5}$ & 0.0    & 59.0  & 97.0 & 97.0  \\ \hline
    \end{tabular}
\caption{Efficiency and robustness rates of the Inexact and Exact Newton methods for three sets of large-scale sparse AVEs varying the condition number of the matrices.}
\label{tabsummary1}
\end{center}}
\end{table}

The numerical experiments of this section allow us to conclude that the inexact Newton method is much more efficient than the exact one when $A$ is a large-scale sparse matrix with a ``moderate'' 
condition number. For matrices with higher condition numbers (order of $10^5$), both methods failed to solve the AVE problem with accuracy $10^{-8}$. Considering a greater tolerance ($10^{-5}$), the exact 
method was the most efficient and robust. This indicates that, at least in our implementation, the inexact Newton method is more sensitive to the increase of the condition number than the exact one.

\subsection{Influence of Density} \label{denseAVE}

The iterative linear equations solvers are suitable for large and sparse matrices, particularly for the routine {\it lsqr}. Therefore, in this section, we investigated the influence of the density of 
matrix $A$ on the performance of the exact and inexact semi-smooth Newton methods. In this third phase, we decided to deal only with relatively ``well-conditioned'' matrices because, according to the 
experiments of the previous section, both methods presented greater robustness on this class of problems. Theoretically, it is expected that the huge efficiency difference presented in 
Figure~\ref{figwellcond} be reduced.

We generated four sets of $200$ AVEs varying the density of matrix $A$. The density of the matrices in each set is approximately equal to: (a) $10\%$, (b) $40\%$, (c) $80\%$ and (d) $100\%$.
The Matlab routine {\it sprand} is efficient to generate sparse matrices. When the input density is high, {\it sprand} requires much processing time. In order to alleviate this cost, we decided to 
decrease the dimension of the matrices. In each set of AVEs we adopted $n$ equal to 3000 for the first two sets and 1500 for the other two. The average condition number of the matrices in each set is 
approximately equal to 11, 35, 17 and 18, respectively. Figure~\ref{figdense} shows the performance profiles, while Table~\ref{tabsummary2} shows the efficiency and robustness rates of the 
Inexact and Exact Newton methods in this four sets of test problems.

\begin{figure}

\begin{minipage}[b]{0.50\linewidth}

\begin{figure}[H]
	\centering
		\includegraphics[scale=0.55]{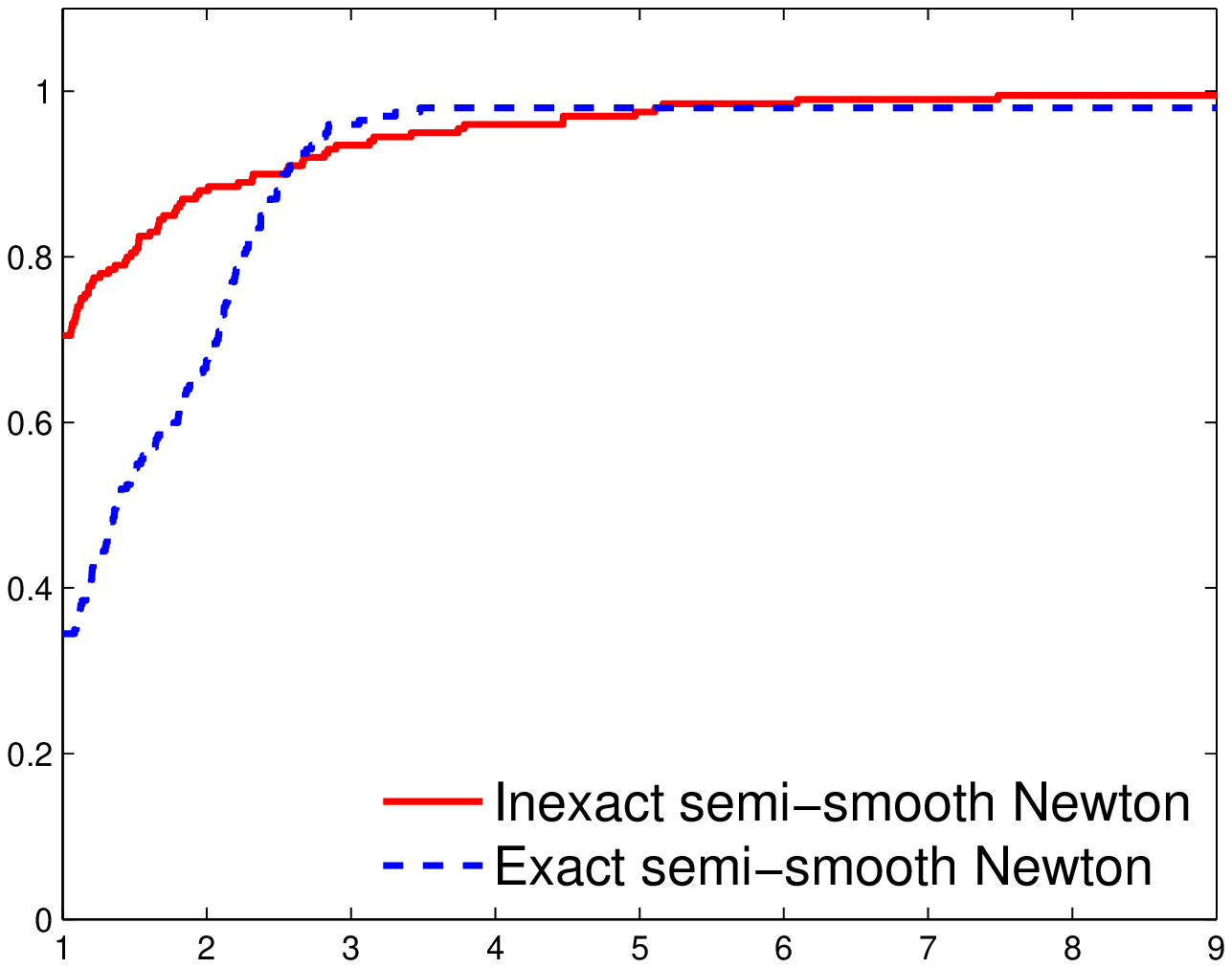}\\
		(a)
\end{figure}

\end{minipage} \hfill
\begin{minipage}[b]{0.50\linewidth}

\begin{figure}[H]
	\centering
		\includegraphics[scale=0.55]{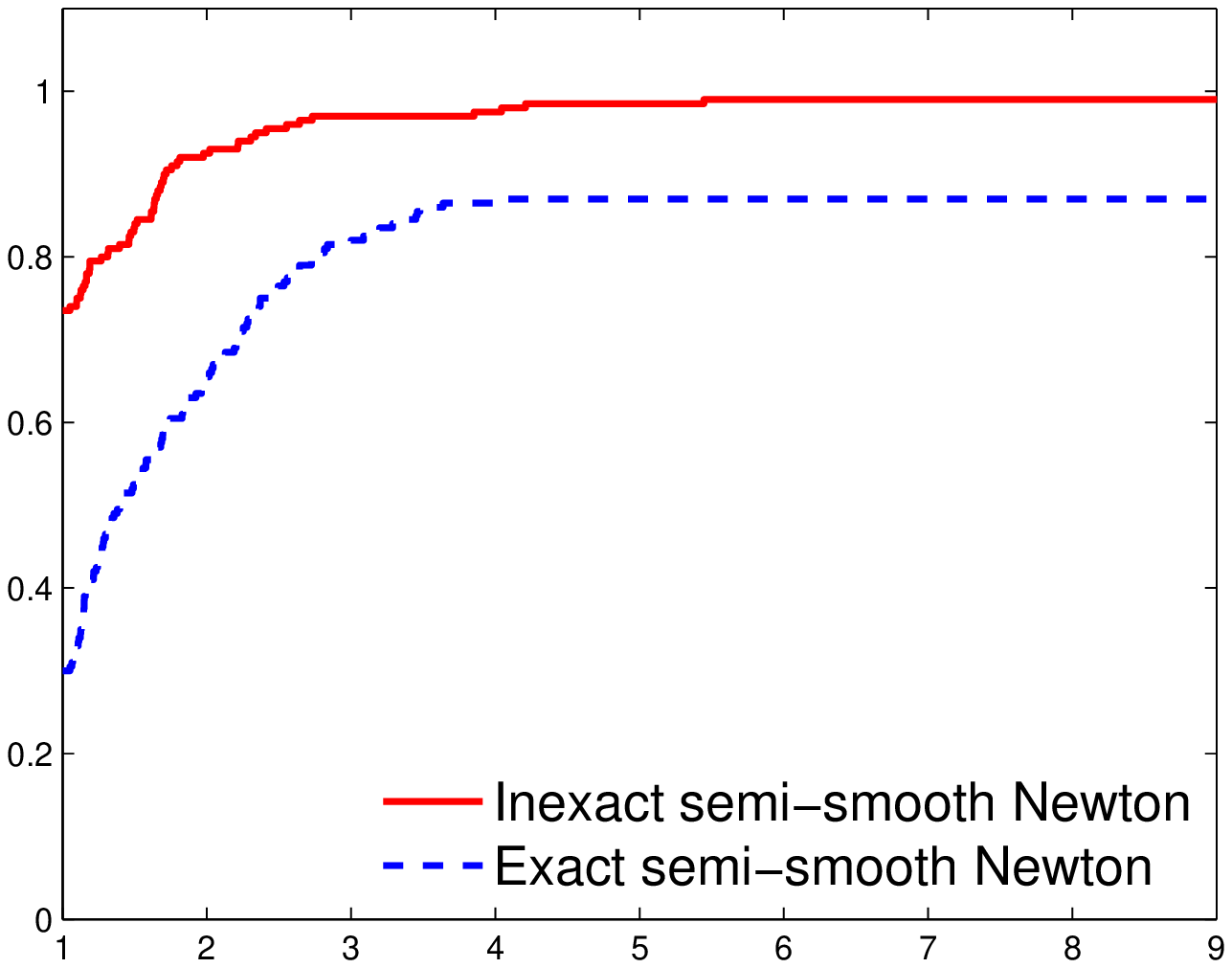}\\
		(b)
\end{figure}

\end{minipage}\hfill

\begin{minipage}[b]{0.50\linewidth}

\begin{figure}[H]
	\centering
		\includegraphics[scale=0.55]{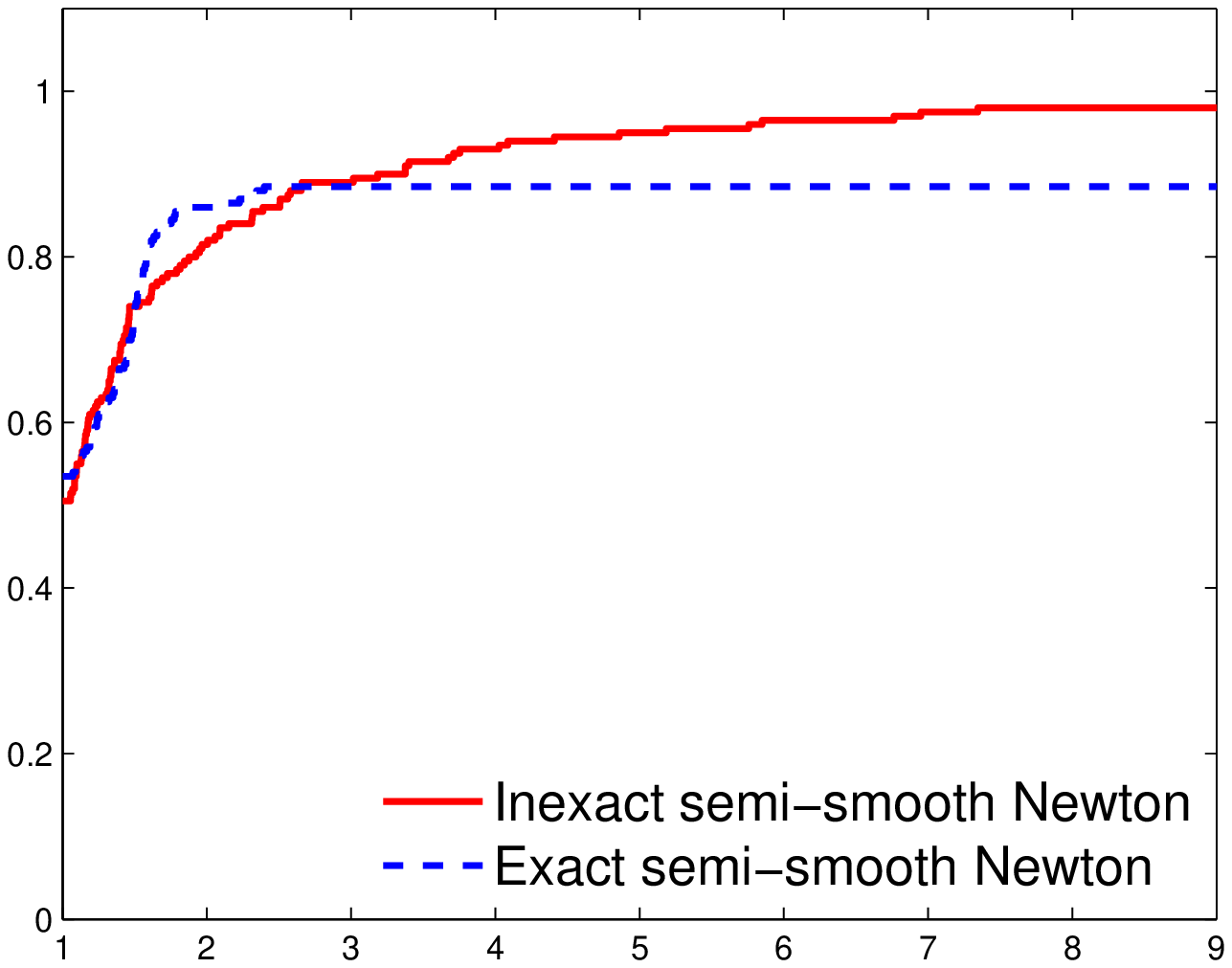}\\
		(c)
\end{figure}

\end{minipage} \hfill
\begin{minipage}[b]{0.50\linewidth}

\begin{figure}[H]
	\centering
		\includegraphics[scale=0.55]{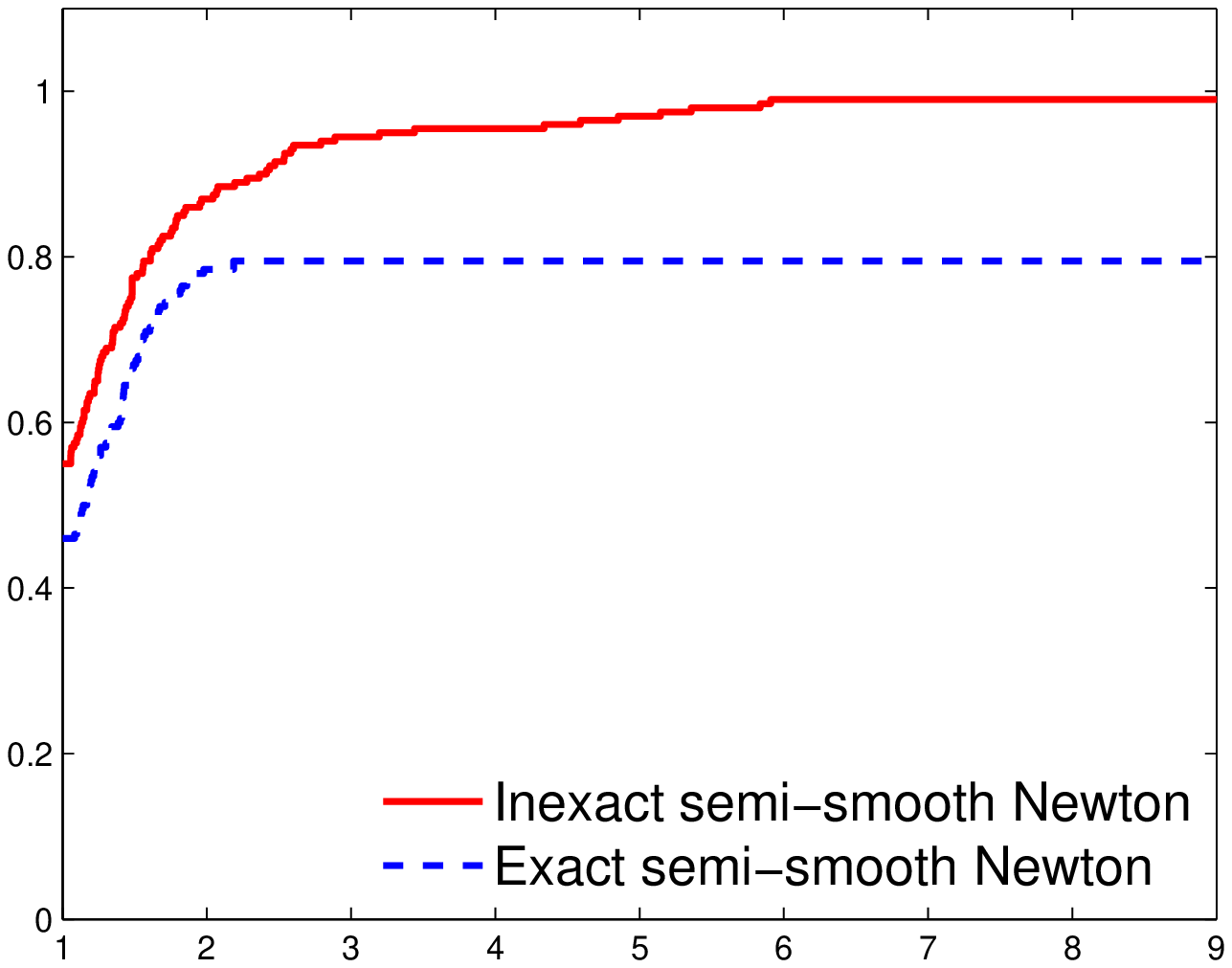}\\
		(d)
\end{figure}

\end{minipage}\hfill

\caption{Performance profile comparing Inexact {\it versus} Exact semi-smooth Newton method for dense AVEs with well-conditioned matrices. Density of $A$ is approximately equal to: (a)~$10\%$, (b)~$40\%$, 
(c)~$80\%$ and (d)~$100\%$.}
\label{figdense}
\end{figure}

\begin{table}[H]
{\footnotesize
\begin{center}
\begin{tabular}{cc|cc|cc|} \cline{3-6}
 & & \multicolumn{2}{|c|}{Inexact Newton method} & \multicolumn{2}{|c|}{Exact Newton method}\\ \hline

\multicolumn{1}{ |c  }{$n$}   & density ($\%$)  & Efficiency ($\%$) & Robustness ($\%$) & Efficiency ($\%$) & Robustness ($\%$)\\  \hline
\multicolumn{1}{ |c  }{3000} & 10  & 70.5 & 99.5  & 34.5  & 98.0    \\
\multicolumn{1}{ |c  }{3000} & 40  & 73.5 & 99.0    & 30.0    & 87.0    \\
\multicolumn{1}{ |c  }{1500} & 80  & 50.5 & 99.5  & 53.5  & 88.5  \\ 
\multicolumn{1}{ |c  }{1500} & 100 & 55.0   & 99.0    & 46.0    & 79.5  \\ \hline
    \end{tabular}
\caption{Efficiency and robustness rates of the Inexact and Exact Newton methods for four sets of AVEs varying the density of the matrices.}
\label{tabsummary2}
\end{center}}
\end{table}

Analyzing Figure~\ref{figdense} and Table~\ref{tabsummary2}, we see that, as expected, the difference between the methods was decreased compared to the case where problems are sparse, 
see Figure~\ref{figwellcond}. For a moderate density 
($10\%$ and $40\%$), the inexact Newton method is more efficient than the exact one (efficiency rate is 70.5$\%$ and 73.5$\%$ against 34.5$\%$ and 30.0$\%$, respectively). Considering full matrices (density 
$80\%$ and $100\%$) both methods showed equivalent efficiency rates: 50.5$\%$ and 55.0$\%$ for the inexact method against 53.5$\%$ and 46.0$\%$ for the exact one, respectively. Independently of density, 
the robustness rate of the inexact Newton method is greater than or equal to 99.0$\%$. On the other hand, as we can see in Table~\ref{tabsummary2}, the robustness rate of the exact method 
decreases according to the increase of the density. This phenomenon is clearly connected with the fact that the greater the density, the more operations are required by the direct method to
to solve a linear equation. Consequently, the exact Newton method is most affected by the accumulation of floating-point errors resulting in lower robustness. Obviously, the number of operations to solve 
a linear equation also depends on the dimension $n$. So, considering the exact Newton method, the robustness rate of 87.0$\%$, in the set of problems where $n= 3000$ and the density is approximately $40\%$, 
is consistent with robustness rate of 88.5$\%$ where the dimension and density are $1500$ and approximately $80\%$, respectively.

The numerical experiments of this section allow us to conclude that, at least in problems where $A$ is a ``well-conditioned'' matrix, 
the inexact Newton method proved to be competitive (in terms of efficiency) with the exact method even if $A$ is a full matrix.
Considering robustness, while the inexact Newton method did not lose performance, the exact method was found to be sensitive with density increase.
\section{Final Remarks} \label{sec:conclusions}
In this work we dealt with the global $Q$-linear convergence of the inexact semi-smooth Newton method for solving AVE in \eqref{eq:vae}. In particular,  a bound for the relative error tolerance to  solve subproblem arises very clearly  in the presented results. The inexact analysis  support the efficient computational  implementations of the exact schemes.  Our implementation  shows the advantage over the exact method in many considered cases, for instance, sparse and large scale problems. We hope that this study serves as a basis for future research on other more efficient variants for solving AVE. We intend to study the semi-smooth Newton method from the point of view of actual implementations and provide comparisons with alternative approaches.  Additional numerical tests indicate that our sufficient condition for convergence can be relaxed, which also deserver to be investigate.

\end{document}